\documentclass[smallextended]{svjour3}
       
\smartqed  

\usepackage{mathptmx}

\usepackage{amsmath,amsfonts,amsthm}

\usepackage{bm}
\usepackage{tikz}
\usetikzlibrary{positioning, calc, chains}
\usetikzlibrary{shapes.misc}
\usetikzlibrary{shapes.symbols}
\usetikzlibrary{shapes.geometric}
\usetikzlibrary{shapes.arrows}
\usetikzlibrary{fit}
\usetikzlibrary{shadows}

\usepackage{pgfplots}
\usepackage{pgfplotstable}

\usepackage{algorithm}
\usepackage[noend]{algorithmic}

\usepackage{subcaption}
\newcommand{\bfb}{\mathbf{b}}
\newcommand{\bff}{\mathbf{f}}
\newcommand{\bfp}{\mathbf{p}}
\newcommand{\bfq}{\mathbf{q}}
\newcommand{\bfv}{\mathbf{v}}
\newcommand{\bfx}{\mathbf{x}}
\newcommand{\bfy}{\mathbf{y}}

\newcommand{\bfalpha}{\bm{\alpha}}
\newcommand{\bfbeta}{\bm{\beta}}
\newcommand{\bfpi}{\bm{\Pi}}
\newcommand{\bftheta}{\bm{\Theta}}
\newcommand{\bfmu}{\bm{\mu}}

\newcommand{\bfzero}{\mathbf{0}}
\newcommand{\bfone}{\mathbf{1}}

\DeclareMathOperator{\diag}{diag}
\DeclareMathOperator{\rank}{rank}
\DeclareMathOperator{\tr}{tr}

\journalname{Numerische Mathematik}

\begin{document}

\title{Bounds-constrained polynomial approximation using the Bernstein basis}       

\author{Larry Allen         \and
        Robert C. Kirby 
}

\institute{Larry Allen \at
              Department of Mathematics, Baylor
    University; 1410 S. 4th Street; Waco, TX 76798-7328 \\
              Tel.: (254) 710-3208\\
              Fax: (254) 710-3659\\
              \email{Larry\_Allen@baylor.edu}           
           \and
           Robert C. Kirby \at
              Department of Mathematics, Baylor
    University; 1410 S. 4th Street; Waco, TX 76798-7328 \\
              Tel.: (254) 710-4846\\
              Fax: (254) 710-3659\\
              \email{Robert\_Kirby@baylor.edu}
}

\date{Received: date / Accepted: date}

\maketitle

\begin{abstract}
  A fundamental problem in numerical analysis and approximation theory is approximating smooth functions by polynomials.  A much harder version under recent consideration is to enforce bounds constraints on the approximating polynomial.  In this paper, we consider the problem of constructiong such approximations using polynomials in the Bernstein basis.  We consider a family of inequality-constrained quadratic programs.  In the univariate case, a quadratic cone constraint allows us to search over all nonnegative polynomials of a given degree.  In both the univariate and multivariate cases, we consider approximate problems with linear inequality constraints.
   Additionally, our method can be modified slightly to include equality constraints such as mass preservation.
\end{abstract}

\section{Introduction}
\label{sec:intro}

A fundamental problem in numerical analysis and approximation theory is to find the best approximation (with respect to a given norm) of a function by polynomials of a fixed degree. A classical version of this problem is to consider functions belonging to a Hilbert space and to find the best approximation in the norm induced by the corresponding inner product~\cite{deutsch2012best}; for example, given a continuous function $f$ over the interval [0,1], find a polynomial $p$ of degree at most $m$ that minimizes $\| f-p\|_{L^2}$. In this setting, there exists a unique solution~\cite{debnath2005hilbert}, and the coefficients of the solution with respect to a chosen basis can be characterized by a linear system involving the mass (or Gram) matrix~\cite{davis1975condition}.  Harder than a Hilbert space setting, one can also consider the best approximation in the norm of uniform convergence.  While the norm is not induced by an inner product, there still exist unique solutions which are characterized by the Chebyshev Alternation Theorem~\cite{cheney1982approximation}; however, the theorem does not provide the solution explicitly~\cite{dehghan2010uniform}, and so solutions are often found for certain types of functions~\cite{eslahchi2009uniform,dehghan2010uniform,jokar2005approximation,lubinsky2003approximation}.

Recently, attention has focused on the much harder version of these problems: approximation subject to bounds constraints; that is, given a continuous function $f$ on [0,1], find a polynomial $q$ of degree at most $m$ with $q(x)\geq 0$ for all $x\in [0,1]$ such that $\|f-q\|_{L^2}$ is minimized~\cite{despres2019positive,despres2020projection,despres2017approximation}. The set of nonnegative polynomials of degree at most $m$ on $[0,1]$ is a closed, convex subset of the space of polynomials of degree at most $m$ (see Proposition~\ref{prop:closestpos}), and so there exists a unique solution to this problem~\cite{debnath2005hilbert}; however, even representing bounds-constrained polynomials presents challenges.  In~\cite{despres2017approximation}, Despr\'{e}s uses the Lukacs Theorem~\cite{szego1939orthogonal} to give a representation of all such polynomials.   This representation is used in~\cite{despres2020projection} to give a nonlinear projection algorithm for approximating bounded functions on an interval in the uniform norm by polynomials that satisfy the same bounds on the interval.  Earlier, Nesterov~\cite{nesterov2000squared} classified nonnegative polynomials over an interval in terms of certain convex cones of their coefficients in the monomial basis.  The coefficients of a polynomial lie in this cone if and only if the polynomial has a certain sum-of-squares type representation that, in the univariate case, is equivalent to nonnegativity.  Nie and Demmel~\cite{nie2006shape}, for example, use this characterization to solve certain shape optimization problems via semidefinite programming.  The cone constraint is equivalent to nonnegativity in the univariate case, but it encodes a type of sum-of-squares property that is only a sufficient condition in the multivariate settings.

Alternatively, one can consider characterizing bounds-constrained polynomials via Bernstein polynomials. Bernstein polynomials were first introduced more than a century ago to give a constructive proof of the Weierstrass approximation theorem~\cite{bernstein1912demo}.  The Bernstein basis forms a nonnegative partition of unity with a geometric decomposition, and the convex hull of a polynomial's coefficients in the Bernstein basis contains that polynomial.  The converse of this last statement is not true, but Bernstein's Theorem~\cite{bernstein1915positive} states that a polynomial satisfies some bounds on an interval if and only if there exists a degree (greater than or equal to the polynomial degree) in which the Bernstein coefficients satisfy the same bounds.  While a general method for computing the exact number of the higher degree is unknown, an upper bound in terms of the minimum of the polynomial and the maximum absolute value of the coefficients in the monomial basis is given in~\cite{reznick2000positive}.  A thorough discussion of certificates of positivity for polynomials in the Bernstein basis (on the simplex as well as the interval) is given in~\cite{leroy2011certificates}. 
In addition, properties of Bernstein polynomials give a special structure of finite element matrices~\cite{kirby2017fast,kirby2012fast} and a structured decomposition of the inverse of matrices related to approximation and interpolation~\cite{allenkirby2020mass,allenkirby2020vandermonde}.  We make use of many of these properties in our present work.

This leads us to consider the problem of approximating a function with a polynomial  whose Bernstein coefficients (perhaps in a fixed higher degree) satisfy the function's bounds; that is, given a continuous function $f$ on [0,1], find a polynomial $q$ of degree at most $m$ such that the degree $n\geq m$ Bernstein coefficients of $q$ are nonnegative and the quantity $\|f-q\|_{L^2}$ is minimized. Similar to the nonnegative approximation, there exists a unique solution to this problem (see Proposition~\ref{prop:closestelev}). While this approximation cannot be better than the best nonnegative approximation, it allows us to frame the approximation problem as a constrained optimization problem with linear inequality constraints.  Although we do not include all nonnegative polynomials as in Nesterov's classification~\cite{nesterov2000squared}, our approach extends seamlessly to a multivariate setting.

By working in the $L^2$ norm rather than the uniform norm, we can make use of certain classical techniques.  We introduce the quadratic cost functional
\begin{equation}
\label{eq:objective}
d_f(q) = \int_0^1 (f-q)^2 dx
\end{equation}
and note that finding the coefficients of the best unconstrained polynomial approximation follows from solving a linear system with a Gram matrix.  We will enforce bounds constraints on the polynomials as linear functions of the Bernstein coefficients, obtaining a quadratic program with linear inequality constraints.  These constraints can be found explicitly and are sufficient conditions for the resulting polynomial to satisfy the desired polynomial bounds.  Exact nonnegativity can be obtained by enforcing quadratic cone constraints. 

One particular application of bounds-constrained approximation comes in the numerical solution of partial differential equations (PDEs), especially hyperbolic equations~\cite{toro2013dynamics}.   Some approaches to limiting in discontinuous Galerkin (DG) methods for hyperbolic PDEs explicitly utilize the geometric properties of Bernstein polynomials to enforce maximum principles and other invariant properties~\cite{hajduk2021monolithic,kuzmin2020subcell}.  These methods are monolithic, with a built-in limiting process.  Here, we pose a problem that is separate from any particular PDE method.  While our method could be used as a limiter with existing DG methods, the problem is interesting and challenging in its own right.  Although our method is focused on preserving a one-sided bound (non-negativity), its extension to two-sided bounds is straight-forward.

An important aspect of our approach compared to the Campos-Pinto, Charles, and Despr\'{e}s (CPCD) algorithm~\cite{despres2019positive} is the straightforward extension of Bernstein polynomials to the simplex.  While the question is natural enough to pose (including, for example, limiters for DG methods), the Lukacs theorem is a univariate result, and generalizations, if possible, are likely to be highly nontrivial~\cite{reznick1992sum}.  The Bernstein basis on a $d$-simplex consists of (suitably scaled) products of barycentric coordinates, maintaining the convex hull property~\cite{LaiSch07}.  This makes it possible to cleanly extend our approach to the simplicial multivariate case.

In this paper, we pose the problem of finding the best $L^2$ approximation of a function subject to positivity in terms of constrained quadratic programs for the approximant's Bernstein coefficients.  In the univariate case, we can pose a quadratically constrained problem over all positive polynomials of a fixed degree.  Additionally, we can approximate this constraint in either the univariate or multivariate case by considering linear inequality constraints on the coefficients.  Methods of semidefinite programming~\cite{vandenberghe1996semidefinite} may be used to efficiently solve these problems, although in the latter case we also give an (exponential) algorithm based on the KKT conditions, which finds the exact solution and also shows how special Bernstein structure can be incorporated in the process.

The paper is organized as follows. In Section~\ref{sec:notation}, we introduce the necessary notation and discuss the existence and uniqueness of solutions of the optimization problem. In Section~\ref{sec:optimization}, we give an algorithm for finding the Bernstein coefficients of the optimal polynomial under linear inequality constraints. In some use cases, it is also desirable to enforce equality constraints such mass preservation; an algorithm for this case is also discussed in Section~\ref{sec:optimization}. Generalizations of these problems to higher dimensions are considered in Section~\ref{sec:higher}. Applications and numerical results are discussed in Section~\ref{sec:num}, and final remarks are given in Section~\ref{sec:fin}.

\section{Notation and existence/uniqueness of solutions}
\label{sec:notation}

For an integer $n\geq 0$, let $\mathcal{P}^n$ denote the space of polynomials of degree at most $n$, and let $\mathcal{P}^{n,+}$ denote the subset of $\mathcal{P}^n$ given by
\begin{equation}
\label{eq:pospoly}
\mathcal{P}^{n,+} = \left\{p\in\mathcal{P}^n\ :\ p(x)\geq 0\ \text{for all}\ x\in[0,1]\right\}.
\end{equation}
For each $0\leq i\leq n$, the $i^{\text{th}}$ Bernstein polynomial of degree $n$ is given by 
\begin{equation}
\label{eq:bernsteinpoly} 
B^n_i(x) = \binom{n}{i}x^i(1-x)^{n-i}. 
\end{equation} 
The Bernstein polynomials form a basis for $\mathcal{P}^n$; that is, every polynomial $p\in\mathcal{P}^n$ can be expressed as 
\begin{equation} 
\label{eq:expansion}
p(x) = \sum_{i=0}^n \bfpi(p)_i B^n_i(x). 
\end{equation} 
We use the notation $\bfpi(p)$ to emphasize the connection between a polynomial $p\in\mathcal{P}^n$ and its vector of coefficients $\bfpi(p)\in\mathbb{R}^{n+1}$; in a similar way, every vector $\bfp\in\mathbb{R}^{n+1}$ generates a polynomial in $\mathcal{P}^n$, which we will denote $\pi(\bfp)$.

If $m\leq n$, then any polynomial expressed in the basis $\{B^m_i(x)\}_{i=0}^m$ can also be expressed in the basis $\{B^n_i(x)\}_{i=0}^n$. We denote by $E^{m,n}$ the $(n+1)\times (m+1)$ matrix that maps the coefficients of the degree $m$ representation to the coefficients of the degree $n$ representation.
It is remarked in~\cite{farouki2000legendre} that the entries of $E^{m,n}$ are given by 
\begin{equation}
\label{eq:elevation}
E^{m,n}_{ij} = \frac{\binom{m}{j}\binom{n-m}{i-j}}{\binom{n}{i}} 
\end{equation} 
with the standard convention that $\binom{n}{i}=0$ whenever $i<0$ or $i>n$.  Note that $E$ is bidiagonal for $n=m+1$ and adds bands with increasing $n$.

Define
\begin{equation}
\label{eq:elevpos}
\mathcal{P}^{m,n} = \left\{p\in\mathcal{P}^m\ :\ E^{m,n}\bfpi(p)\geq\bfzero^n\right\},
\end{equation}
where $\bfzero^n$ denotes the vector of zeros of length $n+1$, and the inequality is understood to be component-wise.

For an integer $n\geq 0$, we let $L^n$ denote the Legendre polynomial (see, for example,~\cite{attar2006orthogonal}) of degree $n$, mapped from its typical home on $[-1,1]$ to $[0,1]$ and scaled so that $L^n(1)=1$ and 
\begin{equation}
\label{eq:legendreL2norm}
\| L^n \|^2_{L^2} = \frac{1}{2n+1}. 
\end{equation} 
It was shown in~\cite{farouki2000legendre} that the Legendre polynomials are represented in the Bernstein basis via 
\begin{equation}
\label{eq:legendrebernstein} 
L^n(x) = \sum_{i=0}^{n} (-1)^{n+i}\binom{n}{i} B^n_i(x);
\end{equation} 
that is, 
\begin{equation}
\label{eq:legendrecoeff} 
\bfpi(L^n)_i = (-1)^{n+i}\binom{n}{i}. 
\end{equation}

Similar to the Bernstein basis, we let $\bftheta(p)$ refer to the vector of coefficients of a polynomial $p$ with respect to the Legendre basis.  We will exploit further connections between these basis later.

The Bernstein mass matrix of degree $n$ is the $(n+1)\times (n+1)$ matrix $M^n$ whose entries are given by 
\begin{equation}
\label{eq:bernsteinmass} 
M^n_{ij} = \int_0^1 B^n_i(x)B^n_j(x) dx. 
\end{equation} 
It was shown in~\cite{kirby2011fast} that the entries can be exactly computed as
\begin{equation} 
\label{eq:massentries}
M^n_{ij} = \binom{n}{i}\binom{n}{j} \frac{(2n-i-j)!(i+j)!}{(2n+1)!}. 
\end{equation}
It was also shown in~\cite{kirby2011fast} that if $m\leq n$, then
\begin{equation}
\label{eq:massreduction}
M^m = \left(E^{m,n}\right)^T M^n E^{m,n}.
\end{equation}

The Bernstein mass matrix connects the $L^2$ topology on the finite-dimensional space to linear algebra. To see this, we observe that if $p,q\in\mathcal{P}^n$, then
\begin{equation}
\label{eq:innprod}
\int_0^1 p(x)q(x) dx = \sum_{i,j=0}^n \bfpi(p)_i\bfpi(q)_j \int_0^1 B^n_i(x) B^n_j(x) dx
= \bfpi(p)^T M^n \bfpi(q).
\end{equation}
In particular, if
\begin{equation}
\label{eq:topology}
\| \bfp \|_{M^n} = \sqrt{ \bfp^T M^n \bfp }
\end{equation}
is the $M^n$-weighted vector norm, then
\begin{equation}
\label{eq:normrelation}
\| p \|_{L^2} = \| \bfpi(p) \|_{M^n}.
\end{equation}
By similar reasoning,
\begin{equation}
\label{eq:Legendrenormrelation}
\| p \|_{L^2} = \| \bftheta(p) \|_{D^n},
\end{equation}
where $D^n = \diag(1,1/3,\dots,1/(2n+1))$.

The previous discussion allows us to show that certain optimization problem over $\mathcal{P}^{m,+}$ and $\mathcal{P}^{m,n}$ have unique solutions. Since $\mathcal{P}^m$ is a Hilbert space, it suffices to show that $\mathcal{P}^{m,+}$ and $\mathcal{P}^{m,n}$ are closed, convex subsets of $\mathcal{P}^m$~\cite{debnath2005hilbert}.

\begin{proposition}
\label{prop:closestpos}
$\mathcal{P}^{m,+}$ is a closed, convex subset of $\mathcal{P}^{m}$.
\end{proposition}

\begin{proof}
If $p(x),q(x)\geq 0$ for all $x\in[0,1]$ and $\nu\in[0,1]$, then
\begin{equation}
\label{eq:convexpos}
\nu p(x) + (1-\nu) q(x)\geq 0\quad \text{for all}\quad x\in[0,1],
\end{equation}
and so $\mathcal{P}^{m,+}$ is convex.

Suppose $\{p_k\}_{k=0}^{\infty}$ is a sequence in $\mathcal{P}^{m,+}$ converging (in the $L^2$ norm) to a polynomial $p\in\mathcal{P}^m$. By (\ref{eq:Legendrenormrelation}), $\{\bftheta(p_k)\}_{k=0}^{\infty}$ converges to $\bftheta(p)$ in the $D^n$ norm. Since all finite-dimensional norms are equivalent, the sequence $\{\bftheta(p_k)\}_{k=0}^{\infty}$ converges to $\bftheta(p)$ in the $\ell^1$ norm. Recall that $|L^j(x)|\leq 1$ for all $x\in [0,1]$, and so
\begin{equation}
\label{eq:pointwise}
\left| p_k(x)-p(x) \right| \leq \| \bftheta(p_k)-\bftheta(p) \|_1\rightarrow 0\quad \text{as}\quad k\rightarrow \infty
\end{equation}
for each $x\in[0,1]$. Therefore, for each $x\in [0,1]$, we have that $\{p_k(x)\}_{k=0}^{\infty}$ is a sequence of nonnegative numbers converging to $p(x)$. This implies that $p\in\mathcal{P}^{m,+}$, and so $\mathcal{P}^{m,+}$ is closed.
\end{proof}

\begin{proposition}
\label{prop:closestelev}
$\mathcal{P}^{m,n}$ is a closed, convex subset of $\mathcal{P}^m$.
\end{proposition}

\begin{proof}
If $p,q\in\mathcal{P}^{m,n}$ and $\nu\in[0,1]$, then
\begin{align*}
\left(E^{m,n}\left[\nu \bfpi(p)+(1-\nu)\bfpi(q)\right]\right)_i &= \nu\left(E^{m,n}\bfpi(p)\right)_i + (1-\nu)\left(E^{m,n}\bfpi(q)\right)_i \\ &\geq 0,
\end{align*}
and so $\mathcal{P}^{m,n}$ is convex.

Suppose $\{p_k\}_{k=0}^{\infty}$ is a sequence in $\mathcal{P}^{m,n}$ converging (in the $L^2$ norm) to a polynomial $p\in\mathcal{P}^m$. By (\ref{eq:normrelation}), the sequence $\{E^{m,n}\bfpi(p_k)\}_{k=0}^{\infty}$ converges to $E^{m,n}\bfpi(p)$ in the $M^n$ norm. Since $[0,\infty)^{n+1}$ is closed, we have that $E^{m,n}\bfpi(p)\in [0,\infty)^{n+1}$, and so $\mathcal{P}^{m,n}$ is closed.
\end{proof}

Because we are working in a Hilbert space, we can use Hilbert orthogonality to restrict the problem of approximating a function to one of approximating a polynomial.  To see this, let $p^* \in \mathcal{P}^m$ be the best unconstrained approximation to some continuous nonnegative $f$ on $[0,1]$. Then $f-p^*$ is orthogonal to $\mathcal{P}^m$.  For any $q \in \mathcal{P}^m$, the Pythagorean Theorem lets us write
\begin{equation}
  \label{eq:pythag}
  \| f - q \|^2 = \| f - p^* \|^2 + \| p^* - q \|^2.
\end{equation}
Since $\| f - p^* \|^2$ is independent of the polynomial $q$, we can minimize the functional $\| p^* - q \|^2$ instead.  This can be more efficient in practice since via~\eqref{eq:topology}, it requires only linear algebra to evaluate rather than evaluating/integrating the continuous function $f$ during the optimization process.

This discussion (combined with Propositions~\ref{prop:closestpos} and~\ref{prop:closestelev}) allows us to state well-posed optimization problems over sets of polynomials:
\begin{theorem}
  For any continuous nonnegative $f$ defined on [0,1] and integer $m \geq 0$, there exists a unique nonnegative $p \in \mathcal{P}^{m,+}$ minimizing
  \(
  \| f - p \|_{L^2}
  \).
\end{theorem}

For univariate polynomials,  membership in $\mathcal{P}^{m,+}$ can be defined in terms of a quadratic cone constraint on the coefficients in the monomial basis, as we describe below.  Hence, the approximation problem may be posed as a quadratically constrained quadratic program and attacked via semidefinite programming.  While the cone constraint possibly could be generalized to a multivariate setting, it is no longer equivalent to nonnegativity.   Anticipating this, \eqref{eq:elevpos} says that membership in $\mathcal{P}^{m,n}$ is given explicitly by nonnegativity of a collection of linear functions of the Bernstein coefficients, and so the approximation problem over this set is a quadratic program with linear inequality constraints.  

\begin{theorem}
  \label{thm:pmn}
  For any continuous nonnegative $f$ defined on [0,1] and integers $m \geq 0$ and $n \geq m$, there exists a unique $p \in \mathcal{P}^{m,n}$ minimizing
  \(
  \| f - p \|_{L^2}.
  \)
\end{theorem}

Enforcing two-sided bounds constraints, such as $f:[0,1] \rightarrow [0,1]$, is also well-posed and requires a simple extension.  In this case, we seek $p\in\mathcal{P}^m$ such that $p\in\mathcal{P}^{m,+}$ and $1-p\in\mathcal{P}^{m,+}$, or as an approximation, $p\in\mathcal{P}^{m,n}$ and $1-p\in\mathcal{P}^{m,n}$.  The resulting sets are still closed, convex subsets of $\mathcal{P}^m$, and so the optimization problems are well-posed.

An important question about these approximation problems is the accuracy of their solutions.
By combining Bernstein's Theorem~\cite{bernstein1915positive} with a result from~\cite{despres2017approximation}, we can give an error estimate for the best approximation in $\mathcal{P}^{m,n}$ provided that $n$ is large enough.

\begin{proposition}
\label{prop:estimate}
There exists an $N\geq m$ such that for $n \geq N$
\begin{equation}
\label{eq:estimate}
\inf_{q\in\mathcal{P}^{m,n}} \| f-q \|_{L^2} \leq 2 \inf_{q\in\mathcal{P}^m} \| f-q \|_{\infty}.
\end{equation}
\end{proposition}

\begin{proof}
By Bernstein's Theorem, there exists an $N\geq m$ such that the best approximation (in the $L^2$ norm) of $f$ in $\mathcal{P}^{m,+}$ belongs to $\mathcal{P}^{m,n}$ for all $n\geq N$. This means that if $n\geq N$, then
\begin{equation}
\label{eq:estineq}
\inf_{q\in\mathcal{P}^{m,n}} \| f-q \|_{L^2} = \inf_{q\in\mathcal{P}^{m,+}} \| f-q \|_{L^2} \leq \inf_{q\in\mathcal{P}^{m,+}} \| f-q \|_{\infty}.
\end{equation}
The result then follows from Theorem 1.2 in~\cite{despres2017approximation}.
\end{proof}

While a method for determining the exact value of $N$ in this theorem is unknown, an upper bound is given in~\cite{reznick2000positive}.  Currently, an error estimate for degrees of elevation less than $N$ is not known.

Now, we discuss certain linear algebraic structure that will be important in presenting a solution of the constrained optimization problem.
In~\cite{allenkirby2020mass}, we showed that $M^n$ admits the spectral decomposition
\begin{equation}
\label{eq:spectral}
M^n = Q^n\Lambda^n\left(Q^n\right)^T,
\end{equation}
where
\begin{equation}
\label{eq:diageigs}
\Lambda^n = \diag(\lambda^n_0, \dots, \lambda^n_n)
\end{equation}
is the diagonal matrix of eigenvalues with 
\begin{equation}
\label{eq:eigs}
\lambda^n_j = \frac{(n!)^2}{(n+j+1)!(n-j)!},
\end{equation}
and
\begin{equation}
\label{eq:eigenvectors}
Q^n[:,j] = \sqrt{(2j+1)\lambda^n_j} E^{j,n}\bfpi(L^j)
\end{equation}
is the orthogonal matrix of eigenvectors.  We also showed (Algorithm 3.1 in~\cite{allenkirby2020mass}) that the matrix $Q^n$ can be constructed explicitly in $\mathcal{O}(n^2)$ operators.

In Section~\ref{sec:optimization}, we will provide an algorithm for finding the coefficients of the optimal polynomial in $\mathcal{P}^{m,n}$. This algorithm provides the degree $n$ representation of a degree $m$ polynomial $q$, and so as part of the algorithm, we must convert the degree $n$ representation of $q$ to its degree $m$ representation. If the vector $\bfy$ contains the degree $n$ representation, then the degree $m$ representation is given by the least squares solution of $E^{m,n}\bfpi(q) = \bfy$; that is,
\begin{equation}
\label{eq:lstsq}
\bfpi(q) = \left( \left(E^{m,n}\right)^T E^{m,n}\right)^{-1} \left(E^{m,n}\right)^T \bfy.
\end{equation}
When $\bfy$ is in the range of $E^{m,n}$, it is also possible to solve $E^{m,n} \bfx = \bfy$ by Gaussian elimination on a rectangular matrix, but the least squares approach is more stable in practice.

The spectral decomposition of $\left(E^{m,n}\right)^T E^{m,n}$ allows us to evaluate~\eqref{eq:lstsq} efficiently.  Since $\left(E^{m,n}\right)^T E^{m,n}$ has a banded structure, it is more efficient to use Gaussian elimination on the least squares system if $n-m$ is small; however, in the interest of stating an algorithm that easily generalizes to higher dimensions, we choose to look at the spectral decomposition.

As a consequence of (\ref{eq:massreduction}) and (\ref{eq:eigenvectors}), we can characterize the eigenvalues and eigenvectors of $\left(E^{m,n}\right)^TE^{m,n}$.

\begin{proposition}
\label{thm:ETEeigs}
\begin{equation}
\label{eq:ETEeigs}
\left(E^{m,n}\right)^TE^{m,n} = Q^m\Sigma^{m,n}\left(Q^m\right)^T,
\end{equation}
where $\Sigma^{m,n} = \diag(\lambda^m_j/\lambda^n_j)_{j=0}^m$ and $Q^m$ is defined in~\eqref{eq:eigenvectors}.
\end{proposition}

\begin{proof}
\begin{align*}
\left(E^{m,n}\right)^TE^{m,n}E^{j,m}\bfpi(L^j) &= \left(E^{m,n}\right)^TE^{j,n}\bfpi(L^j) \\
&= \frac{1}{\lambda^n_j} \left(E^{m,n}\right)^T M^nE^{j,n}\bfpi(L^j) \\
&= \frac{1}{\lambda^n_j} \left(E^{m,n}\right)^T M^nE^{m,n}E^{j,m}\bfpi(L^j) \\
&= \frac{1}{\lambda^n_j} M^mE^{j,m}\bfpi(L^j) \\
&= \frac{\lambda^m_j}{\lambda^n_j} E^{j,m}\bfpi(L^j).
\end{align*}
\end{proof}

\section{Constrained optimization}
\label{sec:optimization}
In this section, we pose optimization problems for finding the Bernstein coefficients of optimal polynomials in $\mathcal{P}^{m,+}$ or $\mathcal{P}^{m,n}$.  In the latter case, we give an exact algorithm based on the KKT conditions for the latter case.
Although the exponential nature of the algorithm makes it impractical for higher degrees, it provides an exact solution process for finding optimal polynomials that can be used as a comparison for other algorithms and demonstrates the role Bernstein structure could play in developing specialized algorithms.

To find the nonnegative polynomial approximation of some $f$ satisfying the hypothesis of Theorem~\ref{thm:pmn}, we first compute the optimal unconstrained approximation $p \in \mathcal{P}^m$ with Bernstein coefficients $\bfp = \bfpi(p)$.  Then, following~\eqref{eq:objective} and~\eqref{eq:normrelation}, we pose the cost functional
\begin{equation}
  d_p(\bfq) = (\bfq-\bfp)^T M^m (\bfq-\bfp).
\end{equation}

We can then pose the approximation problem as a quadratic program with linear inequality constraints:
\begin{equation}
  \label{qp}
  \min_{E^{m,n} \bfq  \geq \bfzero^n} d_p(\bfq) 
\end{equation}
Note that we could also add constraints such as $E^{m,n}\bfq \leq u\bfone^n$, where $\bfone^n$ denotes the vector of ones of length $n+1$, to ensure that that the resulting approximation would satisfy an upper bound of $u$.

Some use cases (e.g. mass conservation) also require the approximating polynomial to preserve the integral average of the original function.  Since
\begin{equation}
\label{eq:integral}
\int_0^1 p(x)dx = \frac{1}{m+1} \sum_{i=0}^m \bfpi(p)_i,
\end{equation}
we can define $h_p(\bfq) = \frac{1}{m+1}\sum_{i=0}^m (\bfpi(p)_i-\bfq_i)$ and consider enforcing $h_p(\bfq) = 0$, in which case we pose the quadratic program with inequality and equality constraints:

\begin{equation}
  \label{qpeq}
  \min_{\stackrel{E^{m,n}\bfq\geq\bfzero^n}{h_p(\bfq) = 0}} d_p(\bfq)
\end{equation}

We can also pose exact nonnegativity via a cone constraint as per~\cite{nesterov2000squared,nie2006shape}, although since we are working with Bernstein polynomials, we must adapt the cone to work on Bernstein coefficients.
Define
\begin{equation}
K^{m,+} = \left\{\bfq\in\mathbb{R}^{m+1}\ :\ \sum_{i=0}^m \bfq_ix^i\geq 0\ \text{for all}\ x\in[0,1]\right\}.
\end{equation} 
Let $S^m$ denote the vector space of $m\times m$ symmetric matrices, let $S^{m,+}$ be the intersection of $S^m$ and the positive semidefinite matrices, and let $H^{m,k}\in S^{m+1}$ be the Hankel matrix given by
\begin{equation}
H^{m,k}_{ij} = \begin{cases} 1, & \text{if}\ i+j=k; \\ 0, & \text{otherwise.} \end{cases}
\end{equation}
Observe that $H^{m,k}$ is the zero matrix if $k<0$ or $k>2m$.
If $m=2\ell$ is even, define the operators $\Omega_0:\mathbb{R}^{m+1}\rightarrow S^{\ell+1}$ and $\Omega_1:\mathbb{R}^{m+1}\rightarrow S^{\ell}$ by
\begin{equation}
\Omega_0(\bfq) = \sum_{k=0}^{2\ell} \bfq_kH^{\ell,k}\quad
\text{and}\quad
\Omega_1(\bfq) = \sum_{k=0}^{2\ell-2} \left[\bfq_{k+1}-\bfq_{k+2}\right]H^{\ell-1,k};
\end{equation}
if $m=2\ell+1$ is odd, define the operators $\Omega_0,\Omega_1:\mathbb{R}^{m+1}\rightarrow S^{\ell+1}$ by
\begin{equation}
\Omega_0(\bfq) = \sum_{k=0}^{2\ell}\bfq_{k+1}H^{\ell,k}\quad
\text{and}\quad
\Omega_1(\bfq) = \sum_{k=0}^{2\ell}\left[ \bfq_k-\bfq_{k+1}\right]H^{\ell,k}.
\end{equation}
Let $\Omega_0^{*}$ and $\Omega_1^{*}$ denote the adjoint of these operators with respect to the inner product $(A,B)=\tr(AB)$ for symmetric matrices.
By~\cite{nesterov2000squared}, we have the following characterization of $K^{m,+}$.

\begin{theorem}[Nesterov]
\label{thm:nesterovcone}
If $m=2\ell$ is even, then
\begin{equation}
K^{m,+} = \{\bfq\in\mathbb{R}^{m+1}\ :\ \bfq = \Omega_0^{*}(A)+\Omega_1^{*}(B)\ \text{for some}\ A\in S^{\ell+1,+}, B\in S^{\ell,+}\};
\end{equation}
if $m=2\ell+1$ is odd, then
\begin{equation}
K^{m,+} = \{\bfq\in\mathbb{R}^{m+1}\ :\ \bfq = \Omega_0^{*}(A)+\Omega_1^{*}(B)\ \text{for some}\ A,B\in S^{\ell+1,+}\}.
\end{equation}
\end{theorem}

A straightforward calculation shows that if $m=2\ell$ is even, then 
\begin{equation}
\Omega_0^{*}(A)_k = \tr(AH^{\ell,k})
\end{equation}
and
\begin{equation}
\Omega_1^{*}(B)_k = \tr(B(H^{\ell-1,k-1}-H^{\ell-1,k-2})).
\end{equation}
Similarly, if $m=2\ell+1$ is odd, then
\begin{equation}
\Omega_0^{*}(A)_k = \tr(AH^{\ell,k-1})
\end{equation}
and
\begin{equation}
\Omega_1^{*}(B)_k = \tr(B(H^{\ell,k}-H^{\ell,k-1})).
\end{equation}
Therefore, we have a description of members of $K^{m,+}$, and so we can consider the optimization problem
\begin{equation}
  \label{qcqp}
\min_{\bfq\in K^{m,+}} d_p(T^m\bfq),
\end{equation}
where $T^m$ denotes the $(m+1)\times (m+1)$ matrix that maps the monomial coefficients of a polynomial $q\in\mathcal{P}^m$ to its corresponding Bernstein coefficients.  This change of basis is triangular, but incredibly ill-conditioned.  On the other hand, the monomial basis is also bad to work with directly.


Although these constrained quadratic programs may be efficiently solved via interior point methods in existing software, with linear constraints we also give a direct algorithm  based on the Karush--Kuhn--Tucker (KKT) conditions~\cite{bazaraa2013kkt} for optimal solutions in nonlinear programming.
We desire to minimize $d_p:\mathbb{R}^{m+1}\rightarrow \mathbb{R}$ subject to linear inequality constraints.  We let $g_i(\bfq) = (E^{m,n}\bfq)_i$ and require $g_i(\bfq)\geq 0$ for $0\leq i\leq n$. The KKT conditions state that there exists $\bfmu\in\mathbb{R}^{n+1}$ such that the optimum $\bfq$ satisfies the stationarity conditions
\begin{equation}
\label{eq:kktstationarity}
\nabla d_p(\bfq)-\sum_{i=0}^n \bfmu_i \nabla g_i(\bfq) = \bfzero^m
\end{equation}
together with the complementary slackness conditions
\begin{equation}
\label{eq:kktcomplementary}
\bfmu_i g_i(\bfq) = 0\quad \text{for each}\quad 0\leq i\leq n
\end{equation}
and the dual feasibility conditions
\begin{equation}
\bfmu_i \geq 0\quad \text{for each}\quad 0\leq i\leq n.
\end{equation}
Such algorithms, although of exponential complexity, do give an exact solution.  They provide a baseline to compare with other optimization algorithms and demonstrate how Bernstein structure may be used within the process.
In Subsection~\ref{ssec:inequality}, we consider the problem of finding $q\in\mathcal{P}^{m,n}$ that solves the inequality-constrained problem, and extend the algorithm to include an equality constraint in Subsection~\ref{ssec:equality}, which only requires a slight modification to~\eqref{eq:kktstationarity}.

\subsection{Finding optimal polynomials via KKT theory}
\label{ssec:inequality}

We begin by minimizing $d_p$ subject to $g_i\geq 0$ for each $0\leq i\leq n$, or equivalently, $E^{m,n} \bfq \geq \bfzero^n$, componentwise. Since
\begin{equation}
\label{eq:objgrad}
\nabla d_p(\bfq) = 2M^m(\bfq-\bfpi(p))
\end{equation}
and
\begin{equation}
\label{eq:ineqgrad}
\nabla g_i(\bfq) = \left( E^{m,n}[i,:]\right)^T,
\end{equation}
the KKT conditions translate to finding $\bfq\in\mathbb{R}^{m+1}$ and $\bfmu\in\mathbb{R}^{n+1}$ that satisfy
\begin{equation}
\label{eq:stationarityneq}
2M^m(\bfq-\bfpi(p))-\left(E^{m,n}\right)^T\bfmu = \bfzero^m,
\end{equation}
\begin{equation}
\label{eq:complementary}
\bfmu_i\left(E^{m,n}\bfq\right)_i=0\quad \text{for each}\quad 0\leq i\leq n,
\end{equation}
and
\begin{equation}
\label{eq:feasibility}
\bfmu_i\geq 0\quad \text{for each}\quad 0\leq i\leq n.
\end{equation}
The complementary slackness conditions give us $2^{n+1}$ cases to consider. For each $J\subseteq\{0,\dots,n\}$, define $I^n_J$ to be the $(n+1)\times (n+1)$ diagonal matrix satisfying
\begin{equation}
\label{eq:Jdiag}
\left(I^n_J\right)_{ii} = \begin{cases} 1, & \text{if}\ i\in J; \\ 0, & \text{if}\ i\not\in J. \end{cases}
\end{equation}
The complementary slackness conditions can then be expressed in the form
\begin{equation}
\label{eq:complementarymat}
I^n_JE^{m,n}\bfq = \bfzero^n\quad \text{and}\quad I^n_{J^c}\bfmu = \bfzero^n,
\end{equation}
where $J^c$ is the relative complement of $J$; that is, $J^c = \{0,\dots, n\}\setminus J$. Therefore, for each $J\subseteq\{0,\dots, n\}$, we can consider the block matrix equation
\begin{equation}
\label{eq:blockneq}
\begin{pmatrix}
M^m & -\frac{1}{2}\left(E^{m,n}\right)^T \\
I^n_JE^{m,n} & I^n_{J^c}
\end{pmatrix}
\begin{pmatrix}
\bfq \\
\bfmu
\end{pmatrix}
=
\begin{pmatrix}
M^m\bfpi(p) \\
\bfzero^n
\end{pmatrix}.
\end{equation}
Since $M^m$ is invertible, we can use the Schur complement to eliminate the variable $\bfq$ and obtain the equation
\begin{equation}
\label{eq:mu}
\left(I^n_{J^c}+\frac{1}{2}I^n_JE^{m,n}\left(M^m\right)^{-1}\left(E^{m,n}\right)^T\right)\bfmu = -I^n_JE^{m,n}\bfpi(p).
\end{equation}
By (\ref{eq:spectral}), 
\begin{equation}
\label{eq:EMinvET}
E^{m,n}\left(M^m\right)^{-1}\left(E^{m,n}\right)^T = U^{m,n}\left(U^{m,n}\right)^T,
\end{equation}
where $U^{m,n}$ is the $(n+1)\times (m+1)$ matrix with columns given by
\begin{equation}
\label{eq:U}
U^{m,n}[:,j] = \sqrt{2j+1}E^{j,n}\bfpi(L^j).
\end{equation}
By defining $W^{m,n} = \frac{1}{2}U^{m,n}\left(U^{m,n}\right)^T$, we can compactly express (\ref{eq:mu}) as
\begin{equation}
\label{eq:murefined}
\left(I^n_{J^c}+I^n_JW^{m,n}\right)\bfmu = -I^n_JE^{m,n}\bfpi(p).
\end{equation}
We remark that $W^{m,n}$ is a rank $m+1$ matrix with nonzero eigenvalues given by $\left(2\lambda^n_j\right)^{-1}$ for $0\leq j\leq m$.

Equation (\ref{eq:murefined}) corresponds to the $|J|\times |J|$ system
\begin{equation}
\label{eq:mufinalneq}
\sum_{j\in J} W^{m,n}_{ij}\bfmu_j = -(E^{m,n}\bfpi(p))_i\quad \text{for each}\quad i\in J.
\end{equation}
We want to find a set $J\subseteq\{0,\dots,n\}$ for which a solution of (\ref{eq:mufinalneq}) exists and satisfies the dual feasibility conditions. In such a case, we observe that (\ref{eq:stationarityneq}) implies that
\begin{equation}
\label{eq:elevatedsolneq}
E^{m,n}\bfq = W^{m,n}\bfmu+E^{m,n}\bfpi(p),
\end{equation}
and so we can check whether the inequality constraints are satisfied. If both the inequality constraints and the dual feasibility conditions are satisfied, then the desired solution is given by
\begin{align*}
\bfq &= \left(\left(E^{m,n}\right)^TE^{m,n}\right)^{-1}\left(E^{m,n}\right)^T\left(W^{m,n}\bfmu+E^{m,n}\bfpi(p)\right) \\
&= U^{m,m}\diag(\lambda^n_0, \dots, \lambda^n_m)\left(U^{m,n}\right)^T\left(W^{m,n}\bfmu+E^{m,n}\bfpi(p)\right),
\end{align*}
where the last equality follows from Theorem~\ref{thm:ETEeigs}.

The previous discussion is summarized in Algorithm~\ref{alg:inequality}. We remark that since solutions to the constrained optimization problem are unique, we can terminate the algorithm once the inequality constraints and dual feasibility conditions are satisfied.

\begin{algorithm}
\caption{Minimizes $d_p(\bfq)$ subject to $g_i(\bfq)\geq 0$ for each $0\leq i\leq n$.}
\label{alg:inequality}
\begin{algorithmic}
\FOR{$J\subset\{0,\dots, n\}$}
\FOR{$i\in J$}
\STATE{$\bfb_i\gets -(E^{m,n}\bfpi(p))_i$}
\FOR{$j\in J$}
\STATE{$A_{ij}\gets W^{m,n}_{ij}$}
\ENDFOR
\ENDFOR
\IF{$\rank(A) = |J|$}
\STATE{$\bfx\gets A^{-1}\bfb$}
\IF{$\bfx\geq\bfzero^{|J|-1}$}
\FOR{$i\gets 0,n$}
\IF{$i\in J$}
\STATE{$\bfmu_i\gets \bfx_i$}
\ELSE
\STATE{$\bfmu_i\gets 0$}
\ENDIF
\ENDFOR
\STATE{$\bfy\gets W^{m,n}\bfmu+E^{m,n}\bfpi(p)$}
\IF{$\bfy\geq\bfzero^n$}
\STATE{$\bfq\gets U^{m,m}\diag(\lambda^n_0,\dots,\lambda^n_m)\left(U^{m,n}\right)^T\bfy$}
\RETURN{$\bfq$}
\ENDIF
\ENDIF
\ENDIF
\ENDFOR
\end{algorithmic}
\end{algorithm}

Since we are iterating over all subsets $J$ of $\{0,\dots,n\}$, the algorithm has exponential complexity. Each iteration requires $2|J|^3/3$ operations to solve for $\bfx$, $(n+1)(2n+2m+3)$ operations to form the vector $\bfy$, and $(n+1)(2m+1)+2(m+1)^2$ operations to form the vector $\bfq$. The vector $\bfy$ is only formed if the entries in $\bfx$ are nonnegative, and the vector $\bfq$ is only formed on the final iteration. Let $N$ denote the number of times the vector $\bfy$ is formed. Even though the algorithm usually terminates early, this results in a worst-case operation count of
\begin{equation}
\label{eq:ineqops0}
N(n+1)(2n+2m+3)+(n+1)(2m+1)+2(m+1)^2+\sum_{J\subset\{0,\dots, n\}}\frac{2|J|^3}{3},
\end{equation}
which can be simplified to
\begin{equation}
\label{eq:ineqops}
N(n+1)(2n+2m+3)+(n+1)(2m+1)+2(m+1)^2+\frac{2^{n-1}}{3}(n+1)^2(n+4).
\end{equation}

\subsection{Enforcing mass preservation}
\label{ssec:equality}

Since
\begin{equation}
\label{eq:eqgrad}
\nabla h_p(\bfq) = -\frac{1}{m+1}\bfone^m,
\end{equation}
we can modify (\ref{eq:stationarityneq}) and find $\bfq\in\mathbb{R}^{m+1}$, $\bfmu\in\mathbb{R}^{n+1}$, and $\nu\in\mathbb{R}$ that satisfy
\begin{equation}
\label{eq:stationarityeq}
2M^m(\bfq-\bfpi(p))-\left(E^{m,n}\right)^T\bfmu - \frac{\nu}{m+1}\bfone^m = \bfzero^m
\end{equation}
together with (\ref{eq:complementary}) and (\ref{eq:feasibility}). Similar to Subsection~\ref{ssec:inequality}, for each $J\subseteq\{0,\dots, n\}$, we consider the block matrix equation
\begin{equation}
\label{eq:blockeq}
\begin{pmatrix}
M^m & -\frac{1}{2}\left(E^{m,n}\right)^T & -\frac{1}{2(m+1)}\bfone^m \\
I^n_JE^{m,n} & I^n_{J^c} & \bfzero^n \\
\frac{1}{m+1}\left(\bfone^m\right)^T & \left(\bfzero^n\right)^T & 0 \\
\end{pmatrix}
\begin{pmatrix}
\bfq \\
\bfmu \\
\nu
\end{pmatrix}
=
\begin{pmatrix}
M^m\bfpi(p) \\
\bfzero^n \\
\frac{1}{m+1}\left(\bfone^m\right)^T \bfpi(p)
\end{pmatrix}.
\end{equation}
We express this as the augmented matrix
\begin{equation}
\label{eq:augment0}
\begin{pmatrix}
M^m & -\frac{1}{2}\left(E^{m,n}\right)^T & -\frac{1}{2(m+1)}\bfone^m & M^m\bfpi(p) \\
I^n_JE^{m,n} & I^n_{J^c} & \bfzero^n & \bfzero^n \\
\frac{1}{m+1}\left(\bfone^m\right)^T & \left(\bfzero^n\right)^T & 0 & \frac{1}{m+1}\left(\bfone^m\right)^T\bfpi(p)
\end{pmatrix}.
\end{equation}
Since $\bfone^m$ is an eigenvector of $M^m$ corresponding to the eigenvalue $\frac{1}{m+1}$ and $E^{m,n}\bfone^m=\bfone^n$, we can reduce the above system to
\begin{equation}
\label{eq:augment1}
\begin{pmatrix}
M^m & -\frac{1}{2}\left(E^{m,n}\right)^T & -\frac{1}{2(m+1)}\bfone^m & M^m\bfpi(p) \\
\bfzero^n\left(\bfzero^m\right)^T & I^n_{J^c}+I^n_J\left(W^{m,n}-\frac{1}{2}\bfone^n\left(\bfone^n\right)^T\right) & \bfzero^n & -I^n_JE^{m,n}\bfpi(p) \\
\left(\bfzero^m\right)^T & \left(\bfone^n\right)^T & 1 & 0
\end{pmatrix}.
\end{equation}
Therefore, we search for $J\subseteq\{0,\dots,n\}$ for which a solution of
\begin{equation}
\label{eq:mufinaleq}
\sum_{j\in J} \left(W^{m,n}_{ij}-\frac{1}{2}\right)\bfmu_j = -(E^{m,n}\bfpi(p))_i\quad \text{for each}\quad i\in J
\end{equation}
exists and satisfies the dual feasibility conditions. In such a case, we have that
\begin{equation}
\label{eq:nu}
\nu = -\sum_{i=0}^n \bfmu_i.
\end{equation}
Since (\ref{eq:stationarityeq}) implies that
\begin{equation}
\label{eq:elevatedsoleq}
E^{m,n}\bfq = W^{m,n}\bfmu+\frac{\nu}{2}\bfone^n+E^{m,n}\bfpi(p),
\end{equation}
we can check whether the inequality constraints are satisfied. If both the inequality constraints and dual feasibility conditions are satisfied, then the desired solution is given by
\begin{equation}
\label{eq:q}
\bfq = U^{m,m}\diag(\lambda^n_0,\dots,\lambda^n_m)\left(U^{m,n}\right)^T\left(W^{m,n}\bfmu+\frac{\nu}{2}\bfone^n+E^{m,n}\bfpi(p)\right).
\end{equation}

The previous discussion is summarized in Algorithm~\ref{alg:equality}. Similar to Subsection~\ref{ssec:inequality}, the algorithm terminates once the inequality constraints and dual feasibility conditions are satisfied.

\begin{algorithm}
\caption{Minimizes $d_p(\bfq)$ subject to $g_i(\bfq)\geq 0$ for each $0\leq i\leq n$ and $h_p(\bfq)=0$.}
\label{alg:equality}
\begin{algorithmic}
\FOR{$J\subset\{0,\dots, n\}$}
\FOR{$i\in J$}
\STATE{$\bfb_i\gets -(E^{m,n}\bfpi(p))_i$}
\FOR{$j\in J$}
\STATE{$A_{ij}\gets W^{m,n}_{ij}-1/2$}
\ENDFOR
\ENDFOR
\IF{$\rank(A) = |J|$}
\STATE{$\bfx\gets A^{-1}\bfb$}
\IF{$\bfx\geq\bfzero^{|J|-1}$}
\FOR{$i\gets 0,n$}
\IF{$i\in J$}
\STATE{$\bfmu_i\gets \bfx_i$}
\ELSE
\STATE{$\bfmu_i\gets 0$}
\ENDIF
\ENDFOR
\STATE{$\nu\gets \sum_{i=0}^n \bfmu_i$}
\STATE{$\bfy\gets W^{m,n}\bfmu-\frac{\nu}{2}\bfone^{n}+E^{m,n}\bfpi(p)$}
\IF{$\bfy\geq\bfzero^n$}
\STATE{$\bfq\gets U^{m,m}\diag(\lambda^n_0,\dots,\lambda^n_m)\left(U^{m,n}\right)^T\bfy$}
\RETURN{$\bfq$}
\ENDIF
\ENDIF
\ENDIF
\ENDFOR
\end{algorithmic}
\end{algorithm}

We remark that Algorithm~\ref{alg:inequality} and Algorithm~\ref{alg:equality} can be combined by introducing a variable $\delta$ that is equal to 1 to enforce the equality constraint and 0 otherwise. We also remark that these algorithms are derived by combining the standard KKT algorithm with Bernstein structure.

We summarize the discussions in Subsection~\ref{ssec:inequality} and Subsection~\ref{ssec:equality} in the following theorem.

\begin{theorem}
\label{thm:opcount}
Given a polynomial $p\in\mathcal{P}^m$, Algorithm~\ref{alg:inequality} exactly computes the Bernstein coefficients of the unique polynomial $q\in\mathcal{P}^{m,n}$ that minimizes the quantity $\|p-q\|_{L^2}$ in at most
\begin{equation}
\label{eq:ineqopcount}
M(n+1)(2n+2m+3)+(n+1)(2m+1)+2(m+1)^2+\frac{2^{n-1}}{3}(n+1)^2(n+4)
\end{equation}
operations, and Algorithm~\ref{alg:equality} exactly computes the Bernstein coefficients of the unique polynomial $q\in\mathcal{P}^{m,n}$ that minimizes the quantity $\| p-q\|_{L^2}$ subject to $\int_0^1 p(x)dx = \int_0^1 q(x)dx$ in at most
\begin{equation}
\label{eq:eqopcount}
N[2(n+1)(n+m+2)+n]+(n+1)(2m+1)+2(m+1)^2+\frac{2^{n-1}}{3}(n+1)^2(n+4)
\end{equation}
operations, where $M$ and $N$ denote the number of times the systems (\ref{eq:mufinalneq}) and (\ref{eq:mufinaleq}) have nonnegative solutions, resepctively.
\end{theorem}

\section{Higher dimensions}
\label{sec:higher}
Bernstein polynomials extend naturally to give a basis for multivariate polynomials of total degree $n$.  Additionally, Bernstein polynomials form a geometrically decomposed partition of unity on the $d$-simplex and maintain the convex hull property.  Analogs of Proposition~\ref{prop:closestpos} and Proposition~\ref{prop:closestelev} hold in the multivariate case, and so we can extend our analysis from Section~\ref{sec:optimization} to nonnegative polynomials on the $d$-simplex. In this section, we discuss the generalization of Algorithm~\ref{alg:inequality} and Algorithm~\ref{alg:equality} to higher dimensions.

Although the linear inequality constraints for nonnegativity used in~\eqref{qp} transfer over naturally to the multivariate case, an exact characterization of nonnegativity via a quadratic or cone constraint is likely not possible.
Such quadratic constraints typically imply that a polynomial is a sum-of-squares.  This can typically be decided via semidefinite programmign in polynomial time, while the general problem if determining nonnegativity is NP-hard in general~\cite{lasserre2007sum}.  Approximating nonnegative polynomials with sums of squares may be possible, but introduces complications beyond the scope of the relatively simple quadratic programs we consider in this paper.

For an integer $d\geq 1$, let $S_d$ be the unit right simplex in $\mathbb{R}^d$. Let $\{\bfv_i\}_{i=0}^d\subset \mathbb{R}^d$ be the vertices of $S_d$, and let $\{\bfb_i\}_{i=0}^d$ denote the barycentric coordinates of $S_d$. Each $\bfb_i$ is an affine map from $\mathbb{R}^d$ to $\mathbb{R}$ such that
\begin{equation}
\label{eq:barycentric}
\bfb_i(\bfv_j) = \begin{cases} 1, & \text{if}\ i=j; \\ 0, & \text{if}\ i\neq j; \end{cases}
\end{equation}
for each vertex $\bfv_j$. Each $\bfb_i$ is nonnegative on $S_d$, and
\begin{equation}
\label{eq:partitionofunity}
\sum_{i=0}^d \bfb_i = 1.
\end{equation}
A multiindex $\bfalpha$ of length $d+1$ is a $(d+1)$-tuple of nonnegative integers, written
\begin{equation}
\label{eq:multiindex}
\bfalpha = (\bfalpha_0,\dots,\bfalpha_d).
\end{equation}
The order of $\bfalpha$, denoted $|\bfalpha|$, is given by
\begin{equation}
\label{eq:order}
|\bfalpha| = \sum_{i=0}^d \bfalpha_i.
\end{equation}
The factorial $\bfalpha!$ of a multiindex $\bfalpha$ is defined by
\begin{equation}
\label{eq:factorial}
\bfalpha! = \prod_{i=0}^d \bfalpha_i.
\end{equation}
The Bernstein polynomials of degree $n$ on the $d$-simplex $S_d$ are defined by
\begin{equation}
B^n_{\bfalpha} = \frac{n!}{\bfalpha!} \prod_{i=0}^d \bfb_i^{\bfalpha_i}.
\end{equation}
The complete set of Bernstein polynomials $\{B^n_{\bfalpha}\}_{|\bfalpha|=n}$ form a basis for polynomials in $d$ variables of total degree at most $n$.

If $m\leq n$, then any polynomial expressed in the basis $\{B^{m}_{\bfalpha}\}_{|\bfalpha|=m}$ can also be expressed in the basis $\{B^n_{\bfalpha}\}_{|\bfalpha|=n}$. We denote by $E^{d,m,n}$ the $\binom{d+n}{d}\times\binom{d+m}{d}$ matrix that maps the coefficients of the degree $m$ representation to the coefficients of the degree $n$ representation. The matrix $E^{d,m,n}$ is sparse and can be applied matrix-free~\cite{kirby2017fast}, if desired.

The Bernstein mass matrix for polynomials in $d$ variables of total degree $n$ is the $\binom{d+n}{d}\times \binom{d+n}{d}$ matrix $M^{d,n}$ whose entries are given by 
\begin{equation}
\label{eq:bernsteindmass} 
M^{d,n}_{\bfalpha\bfbeta} = \int_{S^d} B^n_{\bfalpha}(\bfx)B^n_{\bfbeta}(\bfx) d\bfx. 
\end{equation} 
It was shown in~\cite{kirby2011fast} that the entries can be exactly computed as
\begin{equation} 
\label{eq:dmassentries}
M^{d,n}_{\bfalpha\bfbeta} = \frac{(n!)^2(\bfalpha+\bfbeta)!}{\bfalpha!\bfbeta!(2n+d)!}. 
\end{equation}

Let $\mathcal{L}^{d,0}$ denote the space of constant polynomials, and for each integer $j\geq 1$, let $\mathcal{L}^{d,j}$ denote the space of $d$-variate polynomials that are $L^2$ orthogonal to all polynomials of degree less than $j$. In~\cite{farouki2003orthogonal}, a dimensionally recursive algorithm is given for constructing the Bernstein form of an orthogonal basis for $\mathcal{L}^{d,j}$. Therefore, for each nonnegative integer $j$, we can form the $\binom{d+j}{d}\times \binom{d+j-1}{d-1}$ matrix $L^{d,j}$ whose columns are the Bernstein coefficients of an orthogonal basis for $\mathcal{L}^{d,j}$. The following characterization of the eigenvalues of $M^{d,n}$ can be found in~\cite{kirby2017fast}.

\begin{theorem}
\label{thm:dmasseigs}
The eigenvalues of $M^{d,n}$ are $\{\lambda^{d,n}_j\}_{j=0}^n$, where
\begin{equation}
\label{eq:dmasseigs}
\lambda^{d,n}_j = \frac{(n!)^2}{(n+j+d)!(n-j)!}
\end{equation}
is an eigenvalue of multiplicity $\binom{d+j-1}{d-1}$, and the eigenvectors corresponding to $\lambda^{d,n}_j$ are the columns of $E^{d,j,n}L^{d,j}$.
\end{theorem}

For each $0\leq j\leq n$, define $Q^{d,n,j}$ to be the $\binom{d+n}{d}\times \binom{d+j-1}{d-1}$ matrix whose columns are given by
\begin{equation}
\label{eq:Qd}
Q^{d,n,j}[:,k] = \frac{1}{\| L^{d,j}[:,k] \|_{M^{d,j}}} \left(E^{d,j,n}L^{d,j}\right)[:,k].
\end{equation}
Then Theorem~\ref{thm:dmasseigs} implies that
\begin{equation}
\label{eq:EdMinvET}
E^{d,m,n}\left(M^{d,m,m}\right)^{-1}\left(E^{d,m,n}\right)^T = U^{d,m,n}\left(U^{d,m,n}\right)^T,
\end{equation}
where $U^{d,m,n}$ is $\binom{d+n}{d}\times \binom{d+m}{d}$ matrix given by
\begin{equation}
\label{eq:Ud}
U^{d,m,n} = 
\begin{pmatrix}
Q^{d,n,0} & | & Q^{d,n,1} & | & \cdots & | & Q^{d,n,m}
\end{pmatrix}.
\end{equation}

We now turn our attention to the constrained optimization problem. Define $W^{d,m,n} = \frac{1}{2}U^{d,m,n}\left(U^{d,m,n}\right)^T$. Set $\delta = 1$ to enforce the equality constraints; otherwise, set $\delta=0$. Following similar reasoning as Subsection~\ref{ssec:inequality} and Subsection~\ref{ssec:equality}, we search for a set $J\subset\left\{0,\dots,\binom{d+n}{d}\right\}$ for which the solution of
\begin{equation}
\label{eq:mufinaleqd}
\sum_{j\in J} \left(W^{d,m,n}_{ij}-\frac{d!\delta}{2}\right)\bfmu_j = -(E^{d,m,n}\bfpi(p))_i\quad \text{for each}\quad i\in J
\end{equation}
satisfies the dual feasibility conditions. In such a case, we have that
\begin{equation}
\label{eq:elevatedsoleqd}
E^{d,m,n}\bfq = W^{d,m,n}\bfmu+\frac{\delta\nu}{2}\bfone^{d,n}+E^{d,m,n}\bfpi(p),
\end{equation}
where
\begin{equation}
\label{eq:nud}
\nu = -d!\sum_{|\bfalpha|=n} \bfmu_{\bfalpha},
\end{equation}
and so we can check whether the inequality constraints are satisfied. If both the inequality constraints and dual feasibility conditions are satisfied, then the desired solution is given by
\begin{equation}
\label{eq:qd}
\bfq = U^{d,m,m}\diag(\lambda^{d,n}_0,\dots,\lambda^{d,n}_m)\left(U^{d,m,n}\right)^T\left(W^{d,m,n}\bfmu+\frac{\delta\nu}{2}\bfone^{d,n}+E^{d,m,n}\bfpi(p)\right),
\end{equation}
where each eigenvalue $\lambda^{d,n}_j$ is repeated according to its multiplicity.

The previous discussion is summarized in Algorithm~\ref{alg:optimization}. Similar to Section~\ref{sec:optimization}, the algorithm terminates once the inequality constraints and dual feasibility conditions are satisfied. By generalizing the analysis in Subsection~\ref{ssec:inequality} and Subsection~\ref{ssec:equality}, we can find the worst-case operation count by making the substitutions $m\mapsto \binom{m+d}{d}-1$ and $n\mapsto \binom{d+n}{d}-1$ in (\ref{eq:ineqopcount},\ref{eq:eqopcount}).

\begin{algorithm}
\caption{Minimizes $d_p(\bfq)$ subject to $g_i(\bfq)\geq 0$ for each $0\leq i\leq \binom{d+n}{d}$ and $h_p(\bfq)=0$.}
\label{alg:optimization}
\begin{algorithmic}
\FOR{$J\subset\{0,\dots, \binom{d+n}{d}\}$}
\FOR{$i\in J$}
\STATE{$\bfb_i\gets -(E^{d,m,n}\bfpi(p))_i$}
\FOR{$j\in J$}
\STATE{$A_{ij}\gets W^{d,m,n}_{ij}-\frac{d!\delta}{2}$}
\ENDFOR
\ENDFOR
\IF{$\rank(A) = |J|$}
\STATE{$\bfx\gets A^{-1}\bfb$}
\IF{$\bfx\geq\bfzero^{|J|-1}$}
\FOR{$i\gets 0,n$}
\IF{$i\in J$}
\STATE{$\bfmu_i\gets \bfx_i$}
\ELSE
\STATE{$\bfmu_i\gets 0$}
\ENDIF
\ENDFOR
\STATE{$\nu\gets d!\sum_{|\bfalpha|=n} \bfmu_{\bfalpha}$}
\STATE{$\bfy\gets W^{d,m,n}\bfmu-\frac{\delta\nu}{2}\bfone^{d,n}+E^{d,m,n}\bfpi(p)$}
\IF{$\bfy\geq\bfzero^{d,n}$}
\STATE{$\bfq\gets U^{d,m,m}\diag(\lambda^{d,n}_0,\dots,\lambda^{d,n}_m)\left(U^{d,m,n}\right)^T\bfy$}
\RETURN{$\bfq$}
\ENDIF
\ENDIF
\ENDIF
\ENDFOR
\end{algorithmic}
\end{algorithm}

\section{Numerical results}
\label{sec:num}

In this section, we investigate the accuracy of approximating smooth functions $f:[0,1]\rightarrow [0,1]$ by polynomials whose Bernstein coefficients are nonnegative. Following the discussion in Section~\ref{sec:notation}, we first compute the Bernstein coefficients of the optimal (in the $L^2$ norm) polynomial $p^{*}$ in $\mathcal{P}^m$ via
\begin{equation}
\label{eq:proj}
\bfpi(p^{*}) = U^{m,m}\left(U^{m,m}\right)^T\bff,
\end{equation}
where
\begin{equation}
\label{eq:moments}
\bff_i = \int_0^1 f(x)B^m_i(x) dx.
\end{equation}

We can then find the best approximation in $\mathcal{P}^{m,+}$ or $\mathcal{P}^{m,n}$ by solving the various quadratic programs posed above.  We use the Python package \texttt{cvxpy}~\cite{diamond2016cvxpy} to solve our quadratic programs.  The default solver works well for the linear inequality constraints in~\eqref{qp}, and we use the interface to the Splitting Conic Solver \texttt{SCS}~\cite{ocpb:16,scs} for~\eqref{qcqp}.  We also compare these results to those obtained with Algorithm~\ref{alg:inequality} to solve the linearly constrained problems.  Also, we have implemented the CPCD algorithm for approximation with nonnegative polynomials~\cite{despres2019positive}.

A few simpler approximation schemes also provide a baseline for comparison.  We compute the error in the unconstrained $L^2$ best approximation.  Then,
the \emph{Bernstein polynomial}
\begin{equation}
  B_m(f)(x) = \sum_{i=0}^m f(i/m) B^m_i(x)
\end{equation}
of a continuous function $f$ converges uniformly to $f$ on $[0, 1]$ as $m\rightarrow \infty$, although the rate is at best $\mathcal{O}(m^{-2})$~\cite{LaiSch07}. 

Interpolation of a function by continuous piecewise linear polynomials also preserves bounds constraints, and we do so by using the Bernstein control points as interpolation nodes.  (Note, interpolation into this space exactly corresponds to the $L^2$ projection with mass lumping).  This technique is used as a stage in subcell limiting methods for conservation laws such as~\cite{kuzmin2020subcell}.

We apply these methods to four functions:
\begin{equation}
  \begin{split}
    f_0(x) & = \frac{1}{2} \left( \sin(2\pi x)+1 \right), \\
    f_1(x) & =  0.01 + \frac{x}{x^2+1}, \\
    f_2(x) & = \frac{26}{25} \left( \frac{1}{1+25(2x-1)^2} \right)- \frac{1}{26}, \\
    f_3(x) & = \frac{\pi}{2} + \tan^{-1} \left( 30 \left( x - 1/2 \right) \right).
  \end{split}
\end{equation}
See Figure~\ref{fig:functions} for plots of these.  While all of the functions are smooth, some of the functions are more difficult to approximate by polynomials; for example, $f_2$ and $f_3$ have large derivatives and so require a higher order of approximation to obtain small error than $f_0$ and $f_1$.  Superimposed with the functions are nonnegative polynomial approximations constructed by various techniques; additionaly, we include the best unconstrained polynomial approximation of each function.

We applied each of the approximation algorithms to these problems, showing the results in Figures~\ref{fig:f0acc}--~\ref{fig:f3acc}.  In the left subfigure, we show the results of linear approximation schemes -- the unconstrained $L^2$ projection as well as the bounds-respecting Bernstein polynomial and interpolation into $P^1$.  In the right subfigure, we show the nonlinear approximation algorithms -- the various quadratic programs and the CPCD approximation.

In light of Theorem~\ref{prop:estimate} above and Theorem 1.2 in~\cite{despres2017approximation}, we might expect the error in the quadratic programs to track the unconstrained $L^2$ projection up to some constant factor, but this does not seem to be the case in every situation. For each function we approximate, we see that the cone-constrained quadratic program produces an error very close to the best approximation up to degree 5, but the convergence curves typically flatten off after this.  At higher degree, \texttt{cvxpy} reports an inaccurate solution is obtained, so ill-conditioning or some other numerical difficulty is preventing us from realizing the full theoretical accuracy.

The zero at $x=\tfrac{3}{4}$ in the function $f_0$ presents a major difficulty for our approximating algorithms, as we see in Figure~\ref{fig:f0acc}.  Up to degree 5, the cone-constrained approximation seems to track the $L^2$ approximation quite closely, but after this, \texttt{cvxpy} reports an innaccurate solution.  Moreover, the linearly-constrained optimization algorithms and the CPCD algorithm all struggle, not doing appreciably better than the Bernstein polynomial or $P^1$ interpolation.

The function $f_1$, with no zeros on the interval, presents less difficult for our numerical methods.  Results are shown in Figure~\ref{fig:f1acc}.  The linearly-constrained quadratic programs give error equal to the unconstrained approximation, which indicates that the $L^2$ projection has positive coefficients.  The cone-constrained approximation also matches this until degree 5, after which it becomes inaccurate.  For this problem, all of the nonlinear approximation algorithms fare much better than the Bernstein polynomial and $P^1$ interpolant.

The function $f_2$ has simple zeros at the endpoints of the interval and has large derivatives within the interval and so is more difficult to approximate than $f_1$.  In this case, we see that the $P^1$ interpolant roughly tracks the best $L^2$ approximation.  Like out two previous problems, the cone-constrained approximation does nearly as well as the best $L^2$ approximation, but we lose accuracy as the degree increases.  It and the other nonlinear approximations all fare much better than the Bernstein polynomial.   These results are shown in Figure~\ref{fig:f2acc}, and Figure~\ref{fig:f3acc} shows similar performance for the sigmoid function $f_3(x)$.

A note comparing the results obtained via quadratic programming to those obtained with the CPCD algorithm~\cite{despres2019positive} is also in order.
\emph{A priori}, it is not possible to decide whether CPCD should produce better results than solving~\eqref{qp}.  On one hand, CPCD searches $\mathcal{P}^{m,+}$ rather than $\mathcal{P}^{m,n}$.  On the other hand, it lacks a best approximation property within the search space.  So, for any given function, it may be possible for either algorithm to outperform the other.  We observe that CPCD gives smaller errors for $f_0$ and some higher degrees of approximation for $f_2$ and $f_3$, but larger errors in approximating $f_1$.  We posit that if the best constrained $L^2$ approximation happens to lie in our search space, then solving~\eqref{qp} should win; however, if the best constrained algorithm requires a higher degree of elevation to achieve positive coefficients, then the CPCD algorithm could produce better results.  

Finally, we investigate the accuracy of approximating smooth functions $f:S_2\rightarrow [0,1]$ by polynomials whose Bernstein coefficients are nonnegative. The $L^2$ norm of the errors in the approximations are shown in Figure~\ref{fig:2daccuracy}.  We see that using \texttt{cvxpy} to solve the linearly-constrained problem~\eqref{qp} gives the same results as the KKT-based solver, but (as expected), the convergence is typically slower than for the unconstrained $L^2$ approximation.

To summarize this discussion, we note that the cone-constrained problem~\eqref{qcqp} seems to consistently deliver very good approximations up to degree 5 and that no positivity-enforcing method consistently performs well at high polynomial degree.  On the other hand, the quadratic program with linear constraints generalizes readily to approximations over the simplex, unlike CPCD or the cone-constrained quadratic program.  More work will be needed to improve the numerical accuracy at high degree and to find better ways of enforcing positivity in the multivariate case.

\begin{figure}
\centering
\begin{subfigure}{0.475\linewidth}
\centering
\begin{tikzpicture}[scale=0.725]
\begin{axis}[legend style={at={(0.05,0.05)}, anchor=south west}, xlabel={$x$}, ylabel={$y$}, thick, xmin=0, xmax=1, ymin=-0.1, no marks, xtick={0,0.25,0.5,0.75,1.0}]
\addplot[style=thick] table[x=nodes, y=sinevalues, col sep=comma]{approximation_polys.dat};
\addplot[dashed,style=thick] table[x=nodes, y=sineproj5y, col sep=comma]{approximation_polys.dat};
\addplot[dotted,style=thick] table[x=nodes, y=sineKKT5y, col sep=comma]{approximation_polys.dat};
\addplot[dash dot,style=thick] table[x=nodes, y=sineDespres5y, col sep=comma]{approximation_polys.dat};
\legend{$f_0$,proj5,KKT5,CPCD5}
\end{axis}
\end{tikzpicture}
\caption{$f_0(x) = \frac{\sin(2\pi x)+1}{2}$}
\label{subfig:f0plot}
\end{subfigure}
\hfill
\begin{subfigure}{0.475\linewidth}
\centering
\begin{tikzpicture}[scale=0.725]
\begin{axis}[legend style={at={(0.95,0.05)}, anchor=south east}, xlabel={$x$}, ylabel={$y$}, thick, xmin=0, xmax=1, ymin=0, no marks, xtick={0,0.25,0.5,0.75,1.0}]
\addplot[style=thick] table[x=nodes, y=easyratvalues, col sep=comma]{approximation_polys.dat};
\addplot[dashed,style=thick] table[x=nodes, y=easyratproj5y, col sep=comma]{approximation_polys.dat};
\addplot[dotted, style=thick] table[x=nodes, y=easyratKKT5y, col sep=comma]{approximation_polys.dat};
\addplot[dash dot, style=thick] table[x=nodes, y=easyratDespres5y, col sep=comma]{approximation_polys.dat};
\legend{$f_1$,proj5,KKT5,CPCD5}
\end{axis}
\end{tikzpicture}
\caption{$f_1(x) = 0.01 + x/(1+x^2)$}
\label{subfig:f1plot}
\end{subfigure}
\vskip\baselineskip
\begin{subfigure}{0.475\linewidth}
\centering
  \begin{tikzpicture}[scale=0.725]
\begin{axis}[legend style={at={(0.05,0.95)}, anchor=north west}, xlabel={$x$}, ylabel={$y$}, thick, xmin=0, xmax=1, ymin=-0.1, no marks, xtick={0,0.25,0.5,0.75,1.0}]
\addplot[style=thick] table[x=nodes, y=hardratvalues, col sep=comma]{approximation_polys.dat};
\addplot[dashed,style=thick] table[x=nodes, y=hardratproj5y, col sep=comma]{approximation_polys.dat};
\addplot[dotted,style=thick] table[x=nodes, y=hardratKKT5y, col sep=comma]{approximation_polys.dat};
\addplot[dash dot,style=thick] table[x=nodes, y=hardratDespres5y, col sep=comma]{approximation_polys.dat};
\legend{$f_2$,proj5,KKT5,CPCD5}
\end{axis}
\end{tikzpicture}
\caption{$f_2(x) = \frac{26}{25}\left( \frac{1}{1+25(2x-1)^2}-\frac{1}{26}\right)$}
\label{subfig:f2plot}
\end{subfigure}
\begin{subfigure}{0.475\linewidth}
  \centering
\begin{tikzpicture}[scale=0.725]
\begin{axis}[legend style={at={(0.05,0.95)}, anchor=north west}, xlabel={$x$}, ylabel={$y$}, thick, xmin=0, xmax=1, ymin=-0.1, no marks, xtick={0,0.25,0.5,0.75,1.0}]
\addplot[style=thick] table[x=x, y=f, col sep=comma]{sigmoid_plots.dat};
\addplot[dashed,style=thick] table[x=x, y=Pif, col sep=comma]{sigmoid_plots.dat};
\addplot[dotted,style=thick] table[x=x, y=KKT, col sep=comma]{sigmoid_plots.dat};
\addplot[dash dot,style=thick] table[x=x, y=Despres, col sep=comma]{sigmoid_plots.dat};
\legend{$f_3$,proj5,KKT5,CPCD5}
\end{axis}
\end{tikzpicture}
\caption{$f_3(x) = \frac{\pi}{2}+\tan^{-1}(30(x-1/2))$}
\label{subfig:f3plot}
\end{subfigure}
\caption{Plots of the functions being approximated and their degree 5 polynomial approximations. KKT5 and CPCD5 refer to the finding the degree 5 approximations through the KKT and CPCD algorithms, respectively; proj5 is the unconstrained best degree 5 approximation.}
\label{fig:functions}
\end{figure}
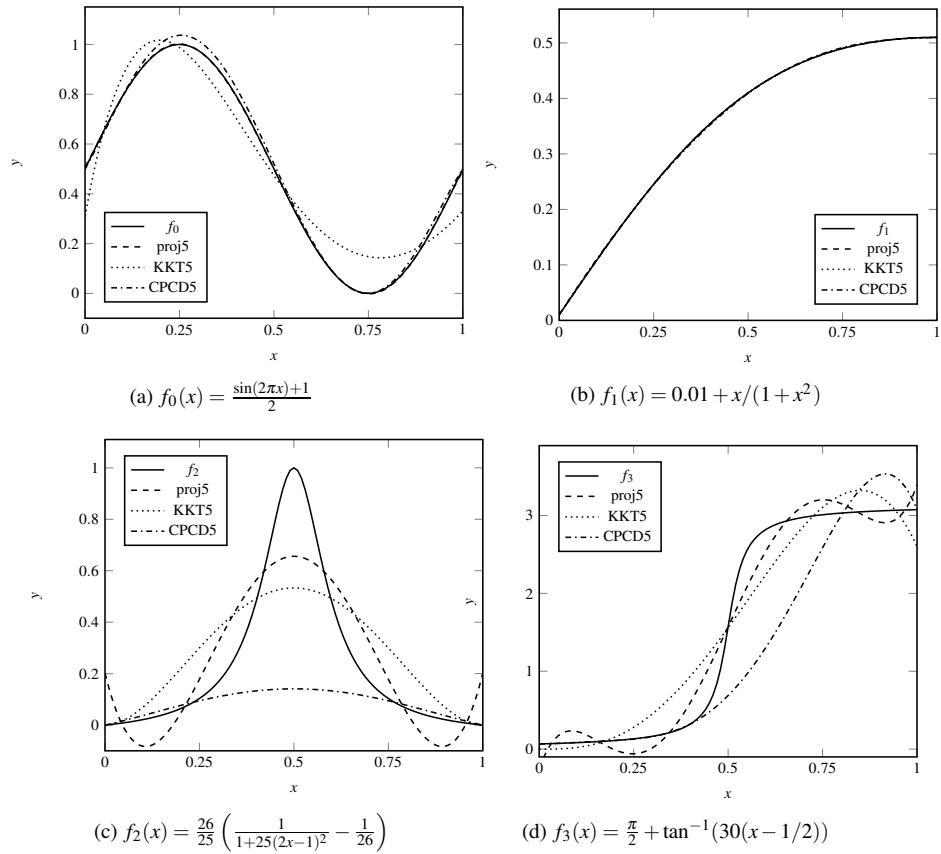

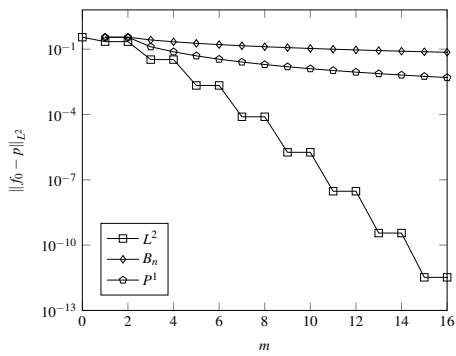
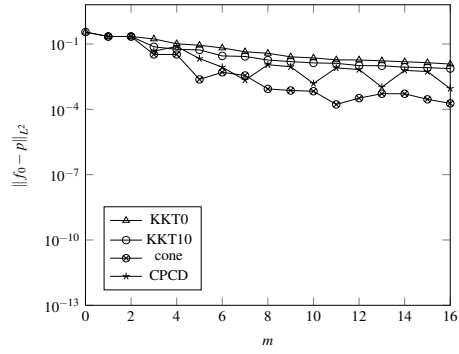
\begin{figure}
\centering
\vskip\baselineskip
\begin{subfigure}{0.475\linewidth}
\centering
\begin{tikzpicture}[scale=0.7]
\begin{semilogyaxis}[legend style={at={(0.05,0.05)}, anchor=south west}, xlabel=$m$, ylabel=$\| f_0-p \|_{L^2}$, xmin=0, xmax=16, ymin=1e-13]
\addplot[mark=square]table[x=m,y=projection,col sep=comma]{closest_poly_v3_info0.dat};
\addplot[mark=diamond] table [x=m, y=Bern, col sep=comma]{extra_experiments_sine.dat};
\addplot[mark=pentagon] table [x=m, y=P1, col sep=comma]{extra_experiments_sine.dat};
\legend{$L^2$, $B_n$, $P^1$}
\end{semilogyaxis}
\end{tikzpicture}
\caption{Linear approximations of $f_0$: unconstrained $L^2$ projection, Bernstein quasi-interpolation, and interpolation onto $P^1$ submesh.}
\label{subfig:f0lin}
\end{subfigure}
\hfill
\begin{subfigure}{0.475\linewidth}
\centering
\begin{tikzpicture}[scale=0.7]
\begin{semilogyaxis}[legend style={at={(0.05,0.05)}, anchor=south west}, xlabel=$m$, ylabel=$\| f_0-p \|_{L^2}$, xmin=0, xmax=16, ymin=1e-13]
\addplot [mark=triangle] table [x=m, y=0, col sep=comma]{closest_poly_v3_info0.dat};
\addplot[mark=o] table [x=m, y=10, col sep=comma]{closest_poly_v3_info0.dat};
\addplot[mark=otimes] table [x=m, y=err, col sep=comma]{cone_op_sine.dat};
\addplot[mark=star] table [x=m, y=error, col sep=comma]{despres_v2_info0.dat};
\legend{KKT0, KKT10,cone,CPCD}
\end{semilogyaxis}
\end{tikzpicture}
\caption{Nonlinear approximations: Solution of~\eqref{qp} with 0 and 10 degrees of elevation (KKT0 and KKT10), solution of quadratic program~\eqref{qcqp} (cone), and the CPCD algorithm}
\label{subfig:f0nonlin}
\end{subfigure}
\caption{Error in approximationing $f_0(x) = \frac{\sin(2\pi x)+1}{2}$ using a variety of methods.}
\label{fig:f0acc}
\end{figure}

\begin{figure}
\centering
\vskip\baselineskip
\begin{subfigure}{0.475\linewidth}
\centering
\begin{tikzpicture}[scale=0.7]
\begin{semilogyaxis}[legend style={at={(0.05,0.05)}, anchor=south west}, xlabel=$m$, ylabel=$\| f_0-p \|_{L^2}$, xmin=0, xmax=16, ymin=1e-13]
\addplot[mark=square]table[x=m,y=projection,col sep=comma]{closest_poly_v3_info1.dat};
\addplot[mark=diamond] table [x=m, y=Bern, col sep=comma]{extra_experiments_easy_rat.dat};
\addplot[mark=pentagon] table [x=m, y=P1, col sep=comma]{extra_experiments_easy_rat.dat};
\legend{$L^2$, $B_n$, $P^1$}
\end{semilogyaxis}
\end{tikzpicture}
\caption{Linear approximations of $f_1$: unconstrained $L^2$ projection, Bernstein quasi-interpolation, and interpolation onto $P^1$ submesh.}
\label{subfig:f1lin}
\end{subfigure}
\hfill
\begin{subfigure}{0.475\linewidth}
\centering
\begin{tikzpicture}[scale=0.7]
\begin{semilogyaxis}[legend style={at={(0.05,0.05)}, anchor=south west}, xlabel=$m$, ylabel=$\| f_0-p \|_{L^2}$, xmin=0, xmax=16, ymin=1e-13]
\addplot [mark=triangle] table [x=m, y=0, col sep=comma]{closest_poly_v3_info1.dat};
\addplot[mark=o] table [x=m, y=10, col sep=comma]{closest_poly_v3_info1.dat};
\addplot[mark=otimes] table [x=m, y=err, col sep=comma]{cone_op_easy_rat.dat};
\addplot[mark=star] table [x=m, y=error, col sep=comma]{despres_v2_info1.dat};
\legend{KKT0, KKT10,cone,CPCD}
\end{semilogyaxis}
\end{tikzpicture}
\caption{Nonlinear approximations: Solution of~\eqref{qp} with 0 and 10 degrees of elevation (KKT0 and KKT10), solution of quadratic program~\eqref{qcqp} (cone), and the CPCD algorithm}
\label{subfig:f1nonlin}
\end{subfigure}
\caption{Error in approximationing $f_1(x)= 0.01 + x/(x^2+1)$ using a variety of methods.  The Bernstein operator and $P^1$ interpolant give very modest decrease in the error, while the solution of~\eqref{qp} with any elevation actually produces the best approximation.  For low polynomial degrees, solving~\eqref{qcqp} also agrees with these, but the solver fails to find an accurate optimal solution after this point.}
\label{fig:f1acc}
\end{figure}
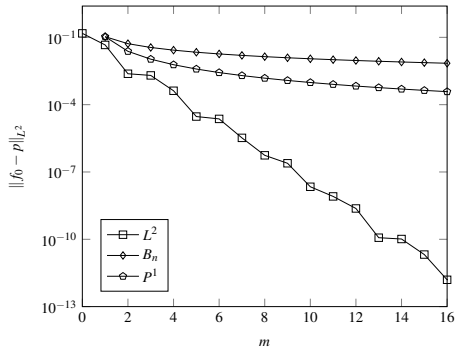
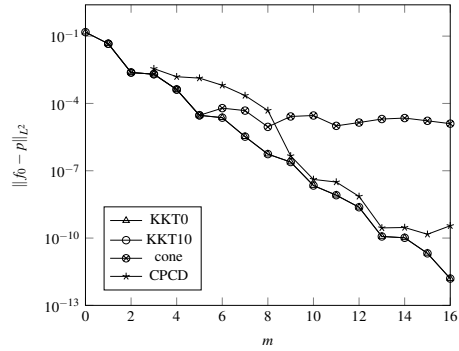

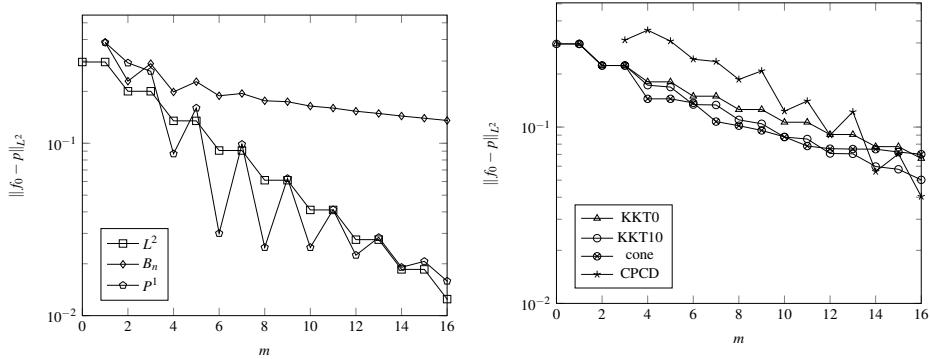
\begin{figure}
\centering
\vskip\baselineskip
\begin{subfigure}{0.475\linewidth}
\centering
\begin{tikzpicture}[scale=0.7]
\begin{semilogyaxis}[legend style={at={(0.05,0.05)}, anchor=south west}, xlabel=$m$, ylabel=$\| f_0-p \|_{L^2}$, xmin=0, xmax=16, ymin=1e-2]
\addplot[mark=square]table[x=m,y=projection,col sep=comma]{closest_poly_v3_info2.dat};
\addplot[mark=diamond] table [x=m, y=Bern, col sep=comma]{extra_experiments_hard_rat.dat};
\addplot[mark=pentagon] table [x=m, y=P1, col sep=comma]{extra_experiments_hard_rat.dat};
\legend{$L^2$, $B_n$, $P^1$}
\end{semilogyaxis}
\end{tikzpicture}
\caption{Linear approximations of $f_2$: unconstrained $L^2$ projection, Bernstein quasi-interpolation, and interpolation onto $P^1$ submesh.}
\label{subfig:f2lin}
\end{subfigure}
\hfill
\begin{subfigure}{0.475\linewidth}
\centering
\begin{tikzpicture}[scale=0.7]
\begin{semilogyaxis}[legend style={at={(0.05,0.05)}, anchor=south west}, xlabel=$m$, ylabel=$\| f_0-p \|_{L^2}$, xmin=0, xmax=16, ymin=1e-2]
\addplot [mark=triangle] table [x=m, y=0, col sep=comma]{closest_poly_v3_info2.dat};
\addplot[mark=o] table [x=m, y=10, col sep=comma]{closest_poly_v3_info2.dat};
\addplot[mark=otimes] table [x=m, y=err, col sep=comma]{cone_op_hard_rat.dat};
\addplot[mark=star] table [x=m, y=error, col sep=comma]{despres_v2_info2.dat};
\legend{KKT0, KKT10,cone,CPCD}
\end{semilogyaxis}
\end{tikzpicture}
\caption{Nonlinear approximations: Solution of~\eqref{qp} with 0 and 10 degrees of elevation (KKT0 and KKT10), solution of quadratic program~\eqref{qcqp} (cone), and the CPCD algorithm}
\label{subfig:f2nonlin}
\end{subfigure}
\caption{Error in approximationing $f_2(x)= (26/25)(1/(1+25(2x-1)^2)-1/26) $ using a variety of methods.  }
\label{fig:f2acc}
\end{figure}

\begin{figure}
\centering
\vskip\baselineskip
\begin{subfigure}{0.475\linewidth}
\centering
\begin{tikzpicture}[scale=0.7]
\begin{semilogyaxis}[legend style={at={(0.05,0.05)}, anchor=south west}, xlabel=$m$, ylabel=$\| f_0-p \|_{L^2}$, xmin=0, xmax=16, ymin=1e-2]
\addplot[mark=square]table[x=m,y=projection,col sep=comma]{closest_poly_v3_info3.dat};
\addplot[mark=diamond] table [x=m, y=Bern, col sep=comma]{extra_experiments_sigmoid.dat};
\addplot[mark=pentagon] table [x=m, y=P1, col sep=comma]{extra_experiments_sigmoid.dat};
\legend{$L^2$, $B_n$, $P^1$}
\end{semilogyaxis}
\end{tikzpicture}
\caption{Linear approximations of $f_3$: unconstrained $L^2$ projection, Bernstein quasi-interpolation, and interpolation onto $P^1$ submesh.}
\label{subfig:f3lin}
\end{subfigure}
\hfill
\begin{subfigure}{0.475\linewidth}
\centering
\begin{tikzpicture}[scale=0.7]
\begin{semilogyaxis}[legend style={at={(0.05,0.05)}, anchor=south west}, xlabel=$m$, ylabel=$\| f_0-p \|_{L^2}$, xmin=0, xmax=16, ymin=1e-2]
\addplot [mark=triangle] table [x=m, y=0, col sep=comma]{closest_poly_v3_info3.dat};
\addplot[mark=o] table [x=m, y=10, col sep=comma]{closest_poly_v3_info3.dat};
\addplot[mark=otimes] table [x=m, y=err, col sep=comma]{cone_op_sigmoid.dat};
\addplot[mark=star] table [x=m, y=error, col sep=comma]{despres_sigmoid.dat};
\legend{KKT0, KKT10,cone,CPCD}
\end{semilogyaxis}
\end{tikzpicture}
\caption{Nonlinear approximations: Solution of~\eqref{qp} with 0 and 10 degrees of elevation (KKT0 and KKT10), solution of quadratic program~\eqref{qcqp} (cone), and the CPCD algorithm}
\label{subfig:f3nonlin}
\end{subfigure}
\caption{Error in approximationing $f_3(x)= \pi/2 + \tan^{-1}(30(x-1/2))$ using a variety of methods.  }
\label{fig:f3acc}
\end{figure}
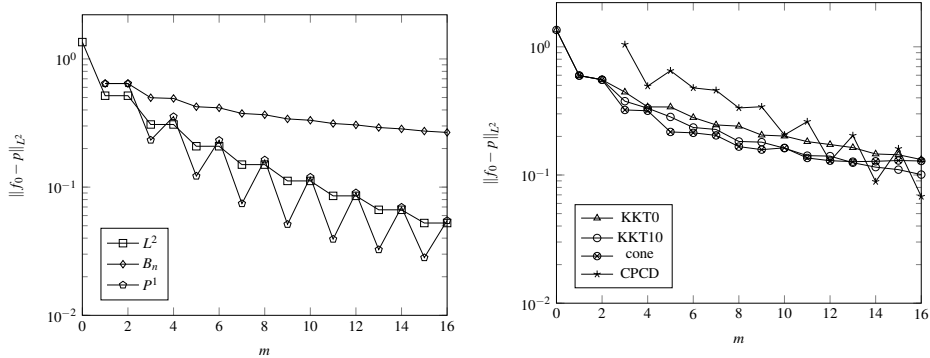

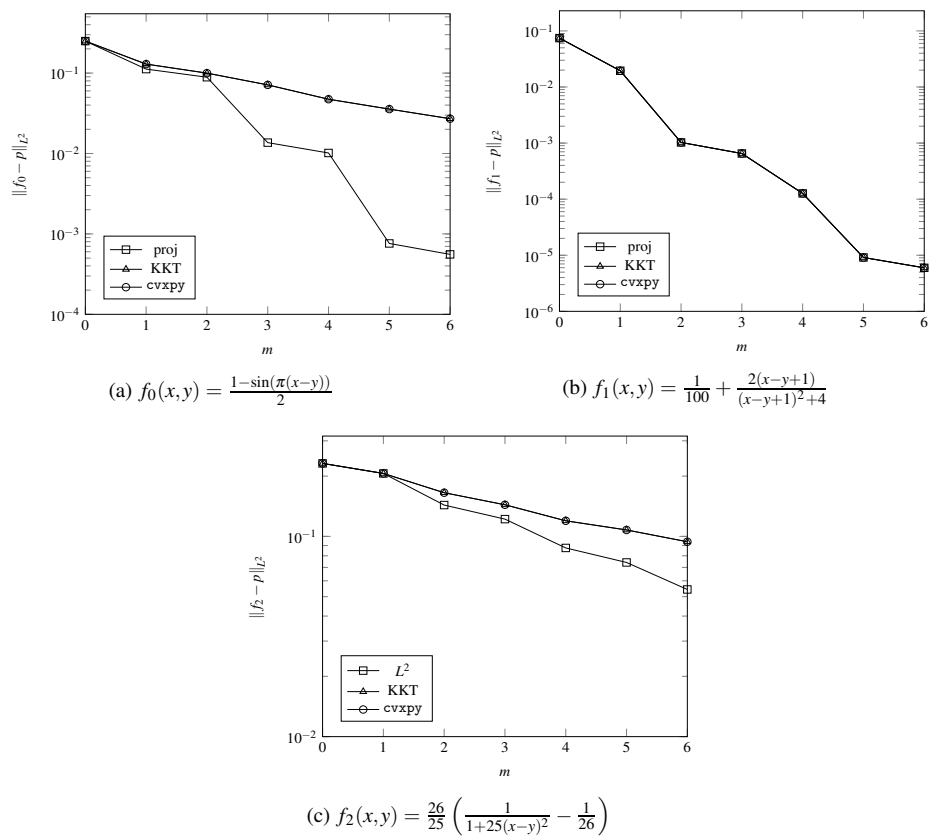
\begin{figure}
\centering
\vskip\baselineskip
\begin{subfigure}{0.475\linewidth}
\centering
\begin{tikzpicture}[scale=0.7]
\begin{semilogyaxis}[legend style={at={(0.05,0.05)}, anchor=south west}, xlabel=$m$, ylabel=$\| f_0-p \|_{L^2}$, xmin=0, xmax=6, ymin=1e-4]
\addplot[mark=square]table[x=m,y=projection,col sep=comma]{sine_2d_info_KKT.dat};
\addplot [mark=triangle] table [x=m, y=KKT, col sep=comma]{sine_2d_info_KKT.dat};
\addplot[mark=o] table [x=m, y=0, col sep=comma]{cp_sine_2d_info.dat};
\legend{proj, KKT, \texttt{cvxpy}}
\end{semilogyaxis}
\end{tikzpicture}
\caption{$f_0(x,y) = \frac{1-\sin(\pi(x-y))}{2}$}
\label{subfig:2df0}
\end{subfigure}
\hfill
\begin{subfigure}{0.475\linewidth}
\centering
\begin{tikzpicture}[scale=0.7]
\begin{semilogyaxis}[legend style={at={(0.05,0.05)}, anchor=south west}, xlabel=$m$, ylabel=$\| f_1-p \|_{L^2}$, xmin=0, xmax=6, ymin=1e-6]
\addplot[mark=square]table[x=m,y=projection,col sep=comma]{easy_rat_2d_info_KKT.dat};
\addplot [mark=triangle] table [x=m, y=KKT, col sep=comma]{easy_rat_2d_info_KKT.dat};
\addplot[mark=o] table [x=m, y=0, col sep=comma]{cp_easy_rat_2d_info.dat};
\legend{proj, KKT, \texttt{cvxpy}}
\end{semilogyaxis}
\end{tikzpicture}
\caption{$f_1(x,y) = \frac{1}{100}+\frac{2(x-y+1)}{(x-y+1)^2+4}$}
\label{subfig:2df1}
\end{subfigure}
\vskip\baselineskip
\begin{subfigure}{0.475\linewidth}
\centering
\begin{tikzpicture}[scale=0.7]
\begin{semilogyaxis}[legend style={at={(0.05,0.05)}, anchor=south west}, xlabel=$m$, ylabel=$\| f_2-p \|_{L^2}$, xmin=0, xmax=6, ymin=1e-2]
\addplot[mark=square]table[x=m,y=projection,col sep=comma]{hard_rat_2d_info_KKT.dat};
\addplot [mark=triangle] table [x=m, y=KKT, col sep=comma]{hard_rat_2d_info_KKT.dat};
\addplot[mark=o] table [x=m, y=0, col sep=comma]{cp_hard_rat_2d_info.dat};
\legend{$L^2$, KKT, \texttt{cvxpy}}
\end{semilogyaxis}
\end{tikzpicture}
\caption{$f_2(x,y) = \frac{26}{25}\left(\frac{1}{1+25(x-y)^2}-\frac{1}{26}\right)$}
\label{subfig:2df2}
\end{subfigure}
\caption{$L^2$ error in approximating $f_j(x,y)$. We use $L^2$ to denote the unconstrained projection of $f_j$ into $\mathcal{P}^m$; we use KKT to denote using Algorithm~\ref{alg:optimization} to find the optimal polynomial in $\mathcal{P}^{m,m}$; and we use \texttt{cvxpy} to denote using \texttt{cvxpy} to find the optimal polynomial in $\mathcal{P}^{m,m}$.}
\label{fig:2daccuracy}
\end{figure}

\section{Conclusions and future work}
\label{sec:fin}
We have proposed new techniques to pose bounds-constrained polynomial approximation over the simplex in terms of quadratic programming.   Quadratic constraints can be used to exactly enforce nonnegativity for univariate polynomials, and the convex hull property of Bernstein polynomials allows us to give explicit sufficient conditions as linear inequality constraints.  These techniques perform comparably to an existing method, and they extend naturally to multivariate polynomials on a simplex of any dimension.   In the future, we hope to make a more thorough study of approximation properties of these algorithms and how to improve the accuracy of solving the quadratically constrained problem.

\clearpage

\bibliographystyle{spmpsci}
\bibliography{references}

\end{document}